\documentclass[11pt]{amsart}
\usepackage{amsthm,amsmath,amssymb,enumerate}
\title[]{Modified K\"ahler-Ricci flow on projective bundles}
\author[]{Ryosuke Takahashi}

\address{Graduate School of Mathematics, Nagoya University, Furo-cho, Chikusa-ku, Nagoya, 464-8602, Japan}

\email{m11036a@math.nagoya-u.ac.jp}


\thanks{This is the author's accepted version. The final publication is available at Springer via http://dx.doi.org/10.1007/s00209-015-1491-y.}

\makeatletter

\@addtoreset{equation}{section}
\makeatother

{
\theoremstyle{definition}
\newtheorem{definition}{Definition}[section]
\newtheorem*{acknowledgements}{Acknowledgements}
}

{
\theoremstyle{plain}
\newtheorem{theorem}{Theorem}[section]

\newtheorem{lemma}{Lemma}[section]
\newtheorem{corollary}{Corollary}[section]
\newtheorem*{maintheorem}{Main theorem}
}

{
\theoremstyle{remark}
\newtheorem{remark}{Remark}[section]
\newtheorem{example}{Example}[section]
}
\begin{document}
%=========Abstract===================================================
\begin{abstract}
On a compact  K\"ahler manifold, a  K\"ahler metric $\omega$ is called Generalized Quasi-Einstein(GQE) if it satisfies the equation 
${\rm Ric} (\omega) - {\mathbb H}{\rm Ric} (\omega) = L_X \omega$ 
for some holomorphic vector field $X$, where ${\mathbb H}{\rm Ric} (\omega)$ denotes the harmonic representative of the Ricci form ${\rm Ric} (\omega)$.
GQE metrics are one of the self-similar solutions of the modified K\"ahler-Ricci flow: 
$\frac{\partial \omega_t}{\partial t} = -{\rm Ric}(\omega_t) + {\mathbb H} {\rm Ric}(\omega _t)$.
In this paper, we propose a method of studying the modified K\"ahler-Ricci flow on special projective bundles, called admissible bundles, from the view point of symplectic geometry. As a result, we can reduce the modified K\"ahler-Ricci flow to a simple PDE with one space variable. Moreover, we study the limiting behavior of the solution in some special cases. 
\end{abstract}
\maketitle
%=========Index======================================================
\tableofcontents
%=========Section 1===================================================
\section{Introduction}
\label{sect:1}
In K{\"a}hler geometry, K{\"a}hler-Einstein metrics are closely related to various types of stabilities, which have been studied by many experts. In order to find K{\"a}hler-Einstein metrics, Tian and Zhu \cite{TZ07} studied the following K{\"a}hler-Ricci flow on an $m$-dimensional Fano manifold $M$:
\begin{equation}
\frac{\partial \omega_t}{\partial t} = - {\rm Ric}(\omega_t) + \omega_t \label{eq:1.1} ,
\end{equation}
where $\omega_t$ is a $t$-dependent K{\"a}hler form and ${\rm Ric}(\omega_t)$ is its Ricci form, which are given by
\[
\begin{cases}
g_{i \bar{j}} = g \left( \frac{\partial}{\partial w^i }, \frac{\partial}{\partial w^{\bar j} } \right) \\
\omega = \sqrt{-1} \sum_{i, j} g_{i \bar{j}} dw^i \wedge dw^{\bar j}
\end{cases}
\]
and
\[
\begin{cases}
r_{i \bar{j}} = - \partial_i \partial_{\bar j} \log ( \det ( g_{k \bar{l}}) ) \\
{\rm Ric} (\omega) = \sqrt{-1} \sum_{i, j} r_{i \bar{j}} dw^i \wedge dw^{\bar j}
\end{cases}
\]
in local holomorphic coordinates $( w^1, \cdots , w^m )$. We assume that the initial metric $\omega_0$ is in $ 2 \pi c_1 (M)$. Then we have $\omega_t \in 2 \pi c_1 (M)$ under the evolution equation \eqref{eq:1.1}.

A K{\"a}hler metric $g$ is called a K{\"a}hler-Ricci soliton if its K{\"a}hler form $\omega \in 2 \pi c_1(M)$ satisfies the equation
\[
{\rm Ric} (\omega ) - \omega = L_X \omega ,
\]
where $L_X$ denotes the Lie derivative with respect to a holomorphic vector field $X$ on $M$. As usual, we denote a K{\"a}hler-Ricci soliton by a pair $(\omega, X)$. If $X = 0$ , this is just a K{\"a}hler-Einstein metric. K{\"a}hler-Ricci solitons are one of the self-similar solutions of K{\"a}hler-Ricci flow. Actually, if we put $\omega_t = ( \exp{(-{\rm Re}(X) \cdot t} ))^* \omega_0$ for any K{\"a}hler-Ricci soliton $( \omega_0, X)$, then $\omega_t$ satisfies the evolution equation \eqref{eq:1.1}. Tian-Zhu \cite{TZ07} proved that if $M$ admits a K{\"a}hler-Ricci soliton $(\omega, X)$ and the initial K{\"a}hler metric is invariant under the action of the one-parameter subgroup generated by ${\rm Im} (X)$, any solution of K{\"a}hler-Ricci flow \eqref{eq:1.1} will converge to the K{\"a}hler-Ricci soliton $(\omega, X)$ in the sense of Cheeger-Gromov.
They also defined a new holomorphic invariant \cite{TZ02}, which is an obstruction to the existence of K{\"a}hler-Ricci solitons  just as the Futaki invariant \cite{Fut83} is an obstruction to the existence of K{\"a}hler-Einstein metrics. By integrating this invariant, they constructed the modified K-energy, which is a functional defined over the space of K{\"a}hler metrics, and convex along geodesics. Then K{\"a}hler-Ricci solitons are cirtical points of this functional. It is conjectured that the existence of a K{\"a}hler-Ricci soliton is equivalent to the strong properness (or coercivity) of the modified K-energy. We can treat all of these materials collectively in Geometric Invariant Theory (GIT), and explain on a formal level why the existence of a K\"ahler-Einstein metric is related to the convergence of K\"ahler-Ricci flow and the coercivity of K-energy (cf. \cite{Don97}). Moerover, on the basis of Hilbert-Mumford criterion in GIT, Donaldson \cite{Don02} introduced a notion of algebro-geometric stability, called K-polystability. By resent works of Chen-Donaldson-Sun \cite{CDS15} and Tian \cite{Tian15}, it was shown that a Fano manifold admits a K\"ahler-Einstein metric if and only if it is K-polystable.

For any polarized manifold, we can give a straightforward extension of K{\"a}hler-Ricci solitons. Let $M$ be a compact K{\"a}hler manifold and $\Omega$ a K{\"a}hler class on $M$. A K{\"a}hler metric $g$ is called a Generalized Quasi-Einstein (GQE) K{\"a}hler metric if its K{\"a}hler form $\omega \in \Omega$ satisfies the equation
\[
{\rm Ric} (\omega) - {\mathbb H}{\rm Ric} (\omega) = L_X \omega ,
\]
where ${\mathbb H}{\rm Ric} (\omega)$ is the harmonic representative of the Ricci form ${\rm Ric} (\omega)$, and $X$ is a holomorphic vector field on $M$. If $X = 0$ , this is just a constant scalar curvature (CSC) K{\"a}hler metric.
Examples of GQE metrics were calculated in \cite{Guan95} and \cite{MT11}, however, the relations between the existence of GQE metrics and stabilities have not been found.

From the above reasons, we want to extend Tian-Zhu's result \cite{TZ07} to the convergence theorem of the modified K{\"a}hler-Ricci flow introduced by Guan \cite{Guan07}:
\begin{equation}
\frac{\partial \omega_t}{\partial t} = -{\rm Ric}(\omega_t) + {\mathbb H} {\rm Ric}(\omega _t) \label{eq:1.2}.
\end{equation}
By a simple calculation, one can check that the evolution equation \eqref{eq:1.2} generalizes \eqref{eq:1.1} for any polarizations, and GQE metrics are one of the self-similar solutions of \eqref{eq:1.2}.
As is the case with K\"ahler-Ricci solitons, it is expected that if we assume that $M$ admits a GQE metric $(\omega , X)$ and the initial K{\"a}hler metric is invariant under the action of the one-parameter subgroup generated by ${\rm Im} (X)$, the long time solution of \eqref{eq:1.2} exists, and will converge to the GQE metric $(\omega, X)$ in the sense of Cheeger-Gromov.

In this paper, we study the evolution equation \eqref{eq:1.2} in a special case: we study \eqref{eq:1.2} on an admissible bundle (cf. \cite{ACGT08}), which is the total space of fiberwise projectification of the direct sum of two projectively-flat holomorphic vector bundles over a compact K{\"a}hler manifold. We assume that $\Omega$ is an admissible K{\"a}hler class whose corresponding polynomial $P(t)$ has exactly one root in the interval $(-1,1)$ and the initial K{\"a}hler metric is an admissible K{\"a}hler metric in $\Omega$. Then the admissible condition is preserved under the flow and \eqref{eq:1.2} can be reduced to the evolution equation
\begin{equation}
\begin{split}
2 \Theta_{\infty} \frac{d \varphi_t}{dt} &= \Theta_{\infty} \Theta_t \varphi _t'' - (\Theta_{\infty} \varphi_t ' )^2 + \frac{P}{p_c} \cdot \Theta_{\infty} \varphi_t ' \\
&+ \left( \Theta_{\infty} \left( \frac{P}{p_c} \right) '  - \left( \frac{P}{p_c} \right) \Theta_{\infty} ' \right) (1+ \varphi_t) \varphi_t \label{eq:1.3}
\end{split}
\end{equation}
for a $t$-dependent smooth function $\varphi_t$ on the interval $[-1,1]$ defined by $\Theta_t=(1+\varphi_t) \Theta_{\infty}$, where $\Theta_t$ (resp. $\Theta_{\infty}$) denotes the function on $[-1,1]$ corresponding to the admissible metric $\omega_t$ (resp. the admissible GQE metric) in $\Omega$. The crucial point is that the evolution equation \eqref{eq:1.3} is basically a heat equation with one space variable. Such a type of equation was first studied by Koiso \cite{Koi90} in the case of anti-canonical polarizations. Thereafter, Guan \cite{Guan07} studied \eqref{eq:1.3} for general polarizations and showed that any convergent solution of \eqref{eq:1.3} decays to $0$ exponentially (on the level of functions defined on $[-1,1]$) under the assumption:
\begin{equation}
\Theta_{\infty} \left( \frac{P}{p_c} \right) '  - \left( \frac{P}{p_c} \right) \Theta_{\infty} ' <0 \;\; \text{on $[-1,1]$}. \label{eq:1.4}
\end{equation}
Actually, the condition \eqref{eq:1.4} is automatically satisfied when $\Omega=2 \pi c_1(M)$, and one can check whether $\Omega$ satisfies the condition \eqref{eq:1.4} or not for many concrete cases. However, it is hard to check for all our cases since $\frac{P}{p_c}$ is not a product of linear factors in general. Our plan is to study  the asymptotic behavior of $P(t)$ as the admissible data of $\Omega$ tends to $0$ and show that the condition \eqref{eq:1.4} is automatically satisfied when the admissible data of $\Omega$ is sufficiently small. Combining with Guan's result, we obtain the following:

\begin{maintheorem}[Theorem~\ref{the:5.2}]
Let $M$ be an $m \left( := \sum_{a \in \hat{\mathcal A}} d_a +1 \right)$-dimensional admissible bundle and $\Omega$ an admissible class on $M$ with the admissible data $\{x_a\}$. We assume that $P(t)$ has exactly one root in the interval $(-1,1)$. Then for any symplectic form defined by \eqref{eq:3.3}, the modified K{\"a}hler-Ricci flow \eqref{eq:1.2} can be reduced to the evolution equation \eqref{eq:1.3} for $\varphi_t$. Moreover, if $|x_a|$ is sufficiently small for all $a \in {\mathcal A}$, any convergent solution $\varphi_t$ of \eqref{eq:1.3} decays to $0$ in exponential order.
\end{maintheorem}
The above theorem agrees with Maschler-T{\o}nnesen's result \cite{MT11}, which says that there exists a unique admissible GQE metric if the admissible data of $\Omega$ is sufficiently small.

Now we describe the content of this paper. In Section \ref{sect:2}, we review the modified K{\"a}hler-Ricci flow studied in \cite{Guan07} and give a few remarks. In Section \ref{sect:3}, we review the fundamental materials about admissible bundles \cite{ACGT08} and define some notations that we will use later. In Section \ref{sect:4}, we relate Maschler-T{\o}nnesen's invariant \cite{MT11} to Tian-Zhu's invariant \cite{TZ02} on admissible bundles (cf. Theorem \ref{the:4.2}). In the course of the proof of Theorem \ref{the:4.2}, the author found the relation between two functions $P$ and $p_c$. Using this relation, we can show that there exists an admissible K\"ahler-Ricci soliton on any admissible bundles polarized by the anti-canonical bundle (cf. Lemma \ref{lem:5.2}), which generalizes the existence result of  a K\"ahler-Ricci soliton on a certain ${\mathbb CP}^1$-bundle obtained by Tian-Zhu \cite{TZ02} to the case when the dimension of the fiber is greater than 1 (cf. Example \ref{exa:5.1}). Lastly, we propose a method of studying the modified K{\"a}hler-Ricci flow on admissible bundles from a symplectic point of view considering the transformations given by $U(1)$-equivariant fiber-preserving diffeomorphisms. By this method, we can regard the evolution equation \eqref{eq:1.2} as a ``symplectic version'' of the modified K{\"a}hler-Ricci flow defined in the moduli space of K{\"a}hler metrics.%=========Acknowledgements===================================================
\begin{acknowledgements}
The author would like to express his gratitude to Professor Ryoichi Kobayashi for his advice on this article, and to the referee for useful suggestions that helped him to improve the original manuscript. The author is supported by Grant-in-Aid for JSPS Fellows Number 25-3077.
\end{acknowledgements}
%=========Section 2===================================================
\section{Modified K{\"a}hler-Ricci flow}
\label{sect:2}
Let $M$ be an $m$-dimensional compact K{\"a}hler manifold and $\Omega$ a K{\"a}hler class on $M$. We consider the following evolution equation:
\begin{equation}
\frac{\partial \omega_t}{\partial t} = -{\rm Ric}(\omega_t) + {\mathbb H} {\rm Ric}(\omega _t) \label{eq:2.1} ,
\end{equation}
where $\omega_t = \sqrt{-1} g_{t \: i \bar{j}} dz^i \wedge dz^{\bar j} \in \Omega$ is a $t$-dependent K{\"a}hler form and ${\mathbb H} {\rm Ric}(\omega _t) = \sqrt{-1} \gamma_{t \: i \bar{j}} dz^i \wedge dz^{\bar j} \in 2 \pi c_1(M)$ is the harmonic representative of the Ricci form ${\rm Ric}(\omega_t) = - \sqrt{-1} \partial \bar{\partial} \log \det g_t = \sqrt{-1} r_{t \: i \bar{j}} dz^i \wedge dz^{\bar j} \in 2 \pi c_1 (M)$.
This is called the modified K{\"a}hler-Ricci flow, which was first introduced in \cite[Section 11]{Guan07}. Because $\frac{\partial [ \omega_t ]}{\partial t} = -2 \pi c_1(M) + 2\pi c_1(M) = 0$, it is clear that if the initial metric $\omega_0$ is in  $\Omega$, then $\omega_t \in \Omega$ for all $t$. Thus the cohomology class of the initial metric is preserved under \eqref{eq:2.1}. If a long time solution of \eqref{eq:2.1} exists and converges to some K{\"a}hler metric, it must be a constant scalar curvature (CSC) K{\"a}hler metric.

However, it is difficult to estimate the behavior of ${\mathbb H} {\rm Ric}(\omega _t)$.
Hence we will study the contraction typed flow instead of \eqref{eq:2.1}:
\begin{equation}
\frac{\partial}{\partial t} {\rm log} \det g_t = -{\rm Scal}(g_t) + \overline{\rm Scal} \label{eq:2.2} ,
\end{equation}
where ${\rm Scal}(g_t) = r_i ^i$ is the scalar curvature of the K{\"a}hler metric $g_t$ and $\overline{\rm Scal} = \gamma_i^i=\frac{ 2 \pi m \cdot c_1(M) \cdot \Omega^{m-1}}{\Omega^m}$ is the average of the scalar curvature.
\begin{remark}
We call ${\rm log} \det g_t $ a \textbf{local Ricci potential}, which is defined in each local coordinate neighborhood. Let $(w^1, \cdots , w^m)$ and $(\tilde{w}^1, \cdots , \tilde{w}^m)$ be $t$-independent local holomorphic coordinate systems. We define the transition map $h = (h^1, \cdots , h^m)$ by $w^i = h^i (\tilde{w}^1, \cdots , \tilde{w}^m)$ for $i = 1, \cdots , m$. If we put $g_{i \bar{j}} = g \left( \frac{\partial}{\partial w^i }, \frac{\partial}{\partial w^{\bar j} } \right)$ and $ \tilde{g}_{i \bar{j}} = g \left( \frac{\partial}{\partial \tilde{w}^i }, \frac{\partial}{\partial \tilde{w}^{\bar j} } \right)$, then we have
\[
\det(\tilde{g}_{i \bar{j}}) = \left| \det \left( \frac{\partial h^i}{\partial \tilde{w}^j} \right) \right|^2 \det (g_{i \bar{j}}) .
\]
Hence local Ricci potentials differ by a $t$-independent function  if we change coordinate systems. Thus $\frac{\partial}{\partial t} {\rm log} \det g_t $ is a function defined over $M$ as long as we treat it in a $t$-independent local coordinate neighborhood.
\end{remark}
The evolution equation \eqref{eq:2.2} is equivalent to \eqref{eq:2.1}. We can check it as follows: Let $\omega_t$ be the solution of \eqref{eq:2.2}, then there exists a $t$-dependent smooth function $f_t$ such that $-{\rm Ric}(\omega_t) + {\mathbb H} {\rm Ric}(\omega _t) = \sqrt{-1} \partial \bar{\partial} f_t$. After taking trace and using the assumption, we get $\Delta_{\partial} f = - \frac{\partial}{\partial t} \log \det g_t$. If we set $g_{i \bar{j}} = g_{0 \: i \bar{j}} + u_{i \bar{j}}$ for some smooth function $u_t$, we have $\Delta_{\partial} f = - g^{i \bar{j}} \cdot \frac{\partial g_{i \bar{j}}}{\partial t} = \Delta_{\partial} \frac{\partial u}{\partial t}$. By the maximum principle, we have $f = \frac{\partial u}{\partial t}$ modulo some $t$-dependent constant. Hence we have $\frac{\partial g_{i \bar{j}}}{\partial t} = -r_{i \bar{j}} + \gamma_{i \bar{j}}$ and this means that $g_t$ is the solution of \eqref{eq:2.1}.
\begin{definition}
\label{def:2.1}
A pair $(\omega, X)$ of a K{\"a}hler form $\omega \in \Omega$ and a holomolphic vector field $X$ is called a Generalized Quasi-Einstein (GQE) K{\"a}hler metric if it satisfies the equation
\begin{equation}
{\rm Ric}(\omega) - {\mathbb H}{\rm Ric}(\omega) = L_X \omega \label{eq:2.3} ,
\end{equation}
where $L_X$ denotes the Lie derivative with respect to $X$.
\end{definition}
If there exists a GQE metric with respect to a holomorphic vector field $X \neq 0$, a long time solution of \eqref{eq:2.1} does not converge. Actually, for any GQE metric $(\omega_0 , X)$, $\omega_t := ( \exp{(-{\rm Re}(X) \cdot t} ))^* \omega_0$ is a solution of \eqref{eq:2.1} and does not converge. In this case, we should add the term $L_X \omega_t$ to the right hand side of \eqref{eq:2.1} and consider the evolution equation:
\begin{equation}
\frac{\partial \omega_t}{\partial t} = -{\rm Ric}(\omega_t) + {\mathbb H} {\rm Ric}(\omega _t) + L_X \omega_t \label{eq:2.4} .
\end{equation}
If a long time solution of \eqref{eq:2.4} exists and converges to some K{\"a}hler metric, it must be a GQE metric with respect to $X$. Generally, it is known that the evolution equation \eqref{eq:2.4} has the unique short time solution \cite[Section 11]{Guan07}.
%=========Section 3===================================================
\section{Admissible bundles}
\label{sect:3}
In this section, we recall special projective bundles called admissible bundles \cite[Section 1]{ACGT08}.
\begin{definition}
\label{def:3.1}
A projective bundle of the form $M= {\mathbb P}(E_0 \oplus E_{\infty}) \rightarrow S$ is called an admissible bundle if it satisfies the following conditions:\\
(1) $S$ has the universal covering $\tilde{S}= \prod_{a \in \mathcal{A}}S_a$ (for a finite set ${\mathcal A} \subset {\mathbb N}$) of simply connected K{\"a}hler manifolds $(S_a,\pm g_a,\pm \omega_a)$ of complex dimensions $d_a$ with $(g_a,\omega_a)$ being pullbacks of tensors on $S$; here $\pm$ means that either $+\omega_a$ or $-\omega_a$ is a K{\"a}hler form which defines a K{\"a}hler metric denoted by $+g_a$ or $-g_a$ respectively.\\
(2) $E_0$ and $E_{\infty}$ are holomorphic plojectively-flat Hermitian vector bundles over $S$ of rank $d_0 +1$ and $d_{\infty} +1$ with $c_1(E_{\infty}) / {\rm rank} E_{\infty} - c_1(E_0)/{\rm rank} E_0 = [\omega_S/ 2\pi]$ and $\omega_S = \sum_{a \in {\mathcal A}} \omega_a$.
\end{definition}

Let $M$ be an admissible bundle. We define several notations and give some remarks that we will use later:
\begin{itemize}
\item We put the index set $\hat{\mathcal A} := \{ a \in {\mathbb N} \cup \{ 0, \infty \} | d_a >0 \}$.
\item $e_0 = {\mathbb P}(E_0 \oplus 0 )$ (resp. $ e_{\infty} = {\mathbb P}(0 \oplus E_{\infty})$) denotes a subbundle of $M$. Then $e_0$ and $e_{\infty}$ are disjoint submanifolds of $M$.
\item ${\mathbb P}(E_0) \rightarrow S$ (resp. ${\mathbb P}(E_{\infty}) \rightarrow S$) is equipped with the fiberwise Fubini-Study metric with the scalar curvature $d_0(d_0 +1)$ (resp. $d_{\infty}(d_{\infty} +1)$), which is denoted by $(g_0, \omega_0)$ (resp. $(-g_{\infty}, -\omega_{\infty})$).
\item Let $\hat{M}$ be the blow-up of $M$ along the set $e_0 \cup e_{\infty}$, and set $ \hat{S} = {\mathbb P} (E_0) \times_S {\mathbb P}(E_{\infty}) \rightarrow S$. Then $\hat{M} \rightarrow \hat{S}$ has a ${\mathbb CP}^1$-bundle structure.
\item We define a $U(1)$-action on $M$ by the canonical $U(1)$-action on $E_0$. Then the Hermitian structures of $E_0$ and $E_{\infty}$ induce the fiberwise moment map $z:M \rightarrow [-1, 1]$ of this $U(1)$-action with critical sets $z^{-1}(1) = e_0$ and $z^{-1}(-1) = e_{\infty}$. 
\item $K$ denotes the infinitesimal generator of the $U(1)$-action on $M$.
\item $\hat{e}_0$ (resp. $\hat{e}_{\infty}$) denotes the exceptional divisor corresponding to the submanifold $e_0$ (resp. $e_{\infty}$), and set $M^0 = M \backslash (e_0 \cup e_{\infty})$. Then $M^0 \rightarrow \hat{S}$ has a ${\mathbb C}^*$-bundle structure.
\item If we regard $M^0$ as an open subset of $\hat{M}$, the restriction of the canonical $U(1)$-action on $\hat{M}$ to $M^0$ coincides with the induced $U(1)$-action from $M$.
\end{itemize}
\begin{definition}
\label{def:3.2}
A K{\"a}hler class $\Omega$ on $M$ is called admissible if there are constants $x_a$, with $x_0 =1$ and $x_{\infty} = -1$, such that
the pullback of $\Omega$ to $\hat{M}$ has the form
\begin{equation}
\Omega = \sum_{a \in \hat{\mathcal A}} [\omega_a]/x_a + \hat{\Xi} \label{eq:3.1} ,
\end{equation}
where $\hat{\Xi}$ is the Poincar\'e dual to $2 \pi [\hat{e}_0 + \hat{e}_{\infty}]$. 
\end{definition}
We can see that any admissible class $\Omega$ has the form
\begin{equation}
\Omega = \sum_{a \in {\mathcal A}} [\omega_a]/x_a + \Xi \label{eq:3.2} ,
\end{equation}
where the pullback of $\Xi$ to $\hat{M}$ is $[\omega_0] - [\omega_{\infty}] + \hat{\Xi}$, i.e., the cohomology class  $[\omega_0] - [\omega_{\infty}]+\hat{\Xi}$
vanishes along the fiber of $\hat e_0 \rightarrow e_0$ and $\hat e_{\infty} \rightarrow e_{\infty}$.
Since $\Omega$ is K{\"a}hler, one can also see that $0 < | x_a | < 1$ for all $a \in {\mathcal A}$ and $x_a$ has the same sign as $g_a$.
Since the blow-up $\hat{M} \rightarrow M$ induces an injective map on cohomology, admissible classes are uniquely determined by the parameters $\{x_a \}$. We call this the \textbf{admissible data} of $\Omega$.

Now we assume that $\pm g_a$ has constant scalar curvature ${\rm Scal}(\pm g_a) = \pm d_a s_a$, where $s_a$ are constants defined in \cite[Section 1.2]{ACGT08}.
\begin{definition}
\label{def:3.3}
Let $\Omega$ be an admissible class with the admissible data $\{x_a \}$.
An admissible K{\"a}hler metric $g$ is the K{\"a}hler metric on $M$ which has the form
\begin{equation}
g = \sum_{a \in \hat{\mathcal A}} \frac{1+ x_a z}{x_a} g_a + \frac{dz^2}{\Theta (z) } + \Theta (z) \theta ^2 , \; \omega = \sum_{a \in \hat{\mathcal A}} \frac{1+ x_a z}{x_a} \omega_a + dz \wedge \theta \label{eq:3.3}
\end{equation}
on $M^0$, where $\theta$ is the connection 1-form $(\theta (K) =1)$ with the curvature $d \theta = \sum_{a \in \hat{\mathcal A}} \omega_a$, and $\Theta$ is a smooth function on $[-1,1]$ satisfying
\begin{equation}
\Theta > 0 \; \text{on} \; (-1,1) , \; \Theta (\pm 1) = 0 \;\; \text{and} \;\; \Theta'(\pm 1) = \mp 2 \label{eq:3.4} .
\end{equation}
The form $\omega$ defined by \eqref{eq:3.3} is a symplectic form, and the compatible complex structure $J$ of $(g, \omega)$ is given by the pullback of the base complex structure and the relation  $Jdz = \Theta \theta$.
\end{definition}
\begin{remark}
\label{rem:3.1}
Using the relation $d \theta = \sum_{a \in \hat{\mathcal A}} \omega_a$, we can check that $\omega$ is closed and $\Omega = [ \omega ]$. Hence $g$ is a K{\"a}hler metric whose K{\"a}hler form $\omega$ belongs to $\Omega$.
\end{remark}
\begin{remark}
\label{rem:3.2}
The defining equation \eqref{eq:3.3} is motivated by the representation of the canonical admissible metric $g_c$ in polar coordinates. In this case, the corresponding function $\Theta_c$ is given by $\Theta_c(z) =1-z^2$, where $g_c$ and $\Theta_c$ will be defined later in this section.
\end{remark}
As will be seen later, the condition \eqref{eq:3.4} is the necessary and sufficient condition to extend a metric $g$ on $M^0$ which has the form \eqref{eq:3.3} to a smooth metric defined on $M$. 
We also use the function
\[
F(z) = \Theta (z) \cdot p_c(z),
\]
where $p_c(z) = \prod_{a \in \hat{\mathcal A}} (1+x_a z)^{d_a}$ is a polynomial of $z$. Then by \eqref{eq:3.4}, we know that $F$ has the condition
\begin{equation}
F > 0 \; \text{on} \; (-1,1) , \; F (\pm 1) = 0 \;\; \text{and} \;\; F'(\pm 1) = \mp 2 p_c (\pm 1) \label{eq:3.5} .
\end{equation}
This is an only necessary condition for $F$, i.e., we can not restore $\Theta$ from $F$ satisfying \eqref{eq:3.5} in general. However, it is possible if $g$ is extremal or GQE (cf. \cite[Section 2.4]{ACGT08} and \cite[Section 4]{MT11}).

Conversely, for any cohomology class $\Omega$ defined by \eqref{eq:3.2}, we can show that $\Omega$ is K{\"a}hler if we assume $0<|x_a| < 1$ and $x_a$ has the same sign as $g_a$, and hence $\Omega$ is admissible. We can prove this by constructing the ``canonical admissible metric'' $g_c$ and the ``canonical symplectic form'' $\omega_c$ belonging to $\Omega$: Let $r_0$ and $r_{\infty}$ be the norm functions induced by the Hermitian metrics on $E_0$ and $E_{\infty}$. Then $z_0 = \frac{1}{2} r_0 ^2$ and $z_{\infty} = \frac{1}{2} r_{\infty}^2$ are fiberwise moment map for the $U(1)$-actions given by the scalar multiplication in $E_0$ and $E_{\infty}$. Let us consider the diagonal $U(1)$-action on $E_0 \oplus E_{\infty}$. Since $U(1)$ acts freely on the level set $z_0 + z_{\infty} = 2$, the restricted metric on this level set descends to the fiberwise Fubini-Study metric on the quotient manifold $M$, which we denote by $(g_{M/S}, \omega_{M/S})$. We extend $(g_{M/S}, \omega_{M/S})$ to a tensor on $M$ by requiring that the horizontal distribution of the induced connection on $M$ is degenerate. Hence $(g_{M/S}, \omega_{M/S})$ is  semi-positive. In order to get a (positive definite) metric on $M$, we set
\[
g_c = \sum_{a \in \mathcal{A}} \frac{1+ x_a z}{x_a} g_a+ g_{M/S}, \; \omega_c = \sum_{a \in {\mathcal A}} \frac{1+ x_a z}{x_a} \omega_a + \omega_{M/S} .
\]
Then $(g_c, \omega_c)$ is a K{\"a}hler metric with respect to the canonical complex structure $J_c$ on $M$. We can see that this metric is admissible and the coresponding function $\Theta_c$ is given by $\Theta_c (z) = 1-z^2$ (cf. \cite[Lemma 1]{ACGT08}). We call this the \textbf{canonical admissible K{\"a}hler metric}.
\begin{remark}
\label{rem:3.3}
In the original paper \cite[Section 1.3 and 1.4]{ACGT08}, admissible classes and admissible metrics are defined by \eqref{eq:3.2} and \eqref{eq:3.3} up to scale respectively because several conditions for metrics (extremal, GQE, etc.) are preserved under scaling of metrics. However, in this paper, the argument of scaling metrics sometimes becomes essential. This is why we define them not up to scale.
\end{remark}
Lastly, we will mention symplectic potentials \cite[Section 1.4]{ACGT08}. As is seen above, admissible metrics with a fixed symplectic form $\omega$ define different complex structures. However, we can regard them as the same complex structure $J_c$ via $U(1)$-equivariant fiber-preserving diffeomorphisms:
a function $u \in C^0 ([-1,1])$ is called a \textbf{symplectic potential} if $u'' (z) = 1/ \Theta (z)$, $u(\pm 1) = 0$ and $u-u_c$ is smooth on $[-1,1]$, where $u_c$ is the \textbf{canonical symplectic potential} defined by
\[
u_c (z) = \frac{1}{2} \left\{ (1-z) \log (1-z) + (1+z) \log (1+z) -2 \log 2 \right\} .
\]
By de l'H\^opital's rule, we can see that there is a one to one correspondence between $u$ and $\Theta$ satisfying \eqref{eq:3.4} (cf. \cite[Lemma 2]{ACGT08}). We can write a K{\"a}hler potential of $\omega$ by means of the symplectic potential $u$ and its fiberwise Legendre transform over $\hat{S}$. Actually, if we put
\begin{equation}
y = u'(z) \;\; \text{and} \;\; h(y) = -u(z) + yz \label{eq:3.6},
\end{equation}
then we obtain $d_J ^c y = \theta$ and $dd_J ^c h (y) = \omega - \sum_{a \in \hat{\mathcal A}} \omega_a / x_a$ on $M^0$.
There are local 1-forms $\alpha$ on $\hat{S}$ such that $\theta = dt + \alpha$, where $t: M^0 \rightarrow {\mathbb R} / 2 \pi {\mathbb Z}$ is an angle function locally defined up to an additive constant.
Let $y_c$ and $h_c$ be the functions corresponding to $u_c$. Since $\exp{(y + \sqrt{-1} t)}$ and $\exp{(y_c + \sqrt{-1} t)}$ give ${\mathbb C}^*$ coordinates on the fibers, there exists $U(1)$-equivariant fiber-preserving diffeomorphism $\Psi$ of $M^0$ such that
\begin{equation}
\Psi^* y = y_c , \; \Psi^* t = t \;\; \text{and} \;\; \Psi^* J = J_c \label{eq:3.7} .
\end{equation}
As $J_c$ and $J$ are integrable complex structures, $\Psi$ extends to a $U(1)$-equivariant diffeomorphism of $M$ leaving fixed any point on $e_0 \cup e_{\infty}$.
Hence $\Psi^* \omega$ is a K{\"a}hler form on $M$ with respect to $J_c$. As $\Psi : (M, J_c) \rightarrow (M, J)$ is biholomorphic, we have $dd_{J_c} ^c h (y_c) = dd_{J_c} ^c h (\Psi^* y) = \Psi^* dd_J ^c h (y) = \Psi^* \omega - \sum_{a \in \hat{\mathcal A}} \omega_a/ x_a$ and $ \Psi^* \omega - \omega = dd^c _{J_c} ( h(y_c) - h_c (y_c))$ on $M^0$, where we remark that the function $h(y_c) - h_c (y_c)$ is extended smoothly on $M$ (cf. \cite[Lemma 3]{ACGT08}). Let ${\mathcal K}_\omega ^{\rm adm}$ be the moduli space of admissible metrics with a fixed symplectic form $\omega$. From the above, we have ${\mathcal K}_\omega ^{\rm adm} = \{ \Theta \; \text{satisfying} \; \eqref{eq:3.4} \} = \{ \text{symplectic potential} \; u\}$.
%=========Section 4===================================================
\section{GQE metrics on Admissible bundles}
\label{sect:4}
Let $M$ be an $m$-dimensional compact K{\"a}hler manifold and $\Omega$ a K{\"a}hler class on $M$.
Let $g$ be a K{\"a}hler metric whose K{\"a}hler form $\omega$ belongs to $\Omega$.
For any holomorphic vector field $V$, we define a complex valued smooth function $\theta_V$ on $M$ by
\begin{equation}
i_V \omega = (\text{harmonic (0,1)-form}) + \sqrt{-1} \bar{\partial} \theta_V \label{eq:4.1} .
\end{equation}
We call the function $\theta_V$ a \textbf{Killing potential} if ${\rm Im} (V)$ is a Killing vector field with respect to $g$, where $i_V$ means the inner product with respect to $V$. The function $\theta_V$ uniquely exists up to an additive constant. We define a real valued smooth function $\kappa$ on $M$ by
\begin{equation}
{\rm Ric}(\omega) - {\mathbb H}{\rm Ric}(\omega) = \sqrt{-1} \partial \bar{\partial} \kappa \label{eq:4.2} .
\end{equation}
The function $\kappa$ is called the \textbf{Ricci potential}. Then we have
\begin{lemma}
\label{lem:4.1}
A K{\"a}hler metric $g$ is a GQE metric with respect to a holomorphic vector field $X$ if and only if its Ricci potential $\kappa$ satisfies the equation $ \kappa = \theta_X$ up to an additive constant.
\end{lemma}
\begin{proof}
Applying $d$ to the both hand sides of \eqref{eq:4.1}, we get $L_V \omega = \sqrt{-1} \partial \bar{\partial} \theta_V$. Combining this with \eqref{eq:2.3} and \eqref{eq:4.2}, and using the maximum principle, we have the desired result. 
\end{proof}
Taking the trace of the both hand sides of \eqref{eq:4.2}, we have
\begin{equation}
{\rm Scal}_g (\omega) - \overline{\rm Scal} = - \Delta_{\partial} \kappa \label{eq:4.3} .
\end{equation}
Now we will consider the case when $\Omega = 2 \pi c_1(M)$ for a moment. Since ${\mathbb H}{\rm Ric} = \omega$, \eqref{eq:2.3} becomes
\begin{equation}
{\rm Ric}(\omega) - \omega = L_X \omega \label{eq:4.4} ,
\end{equation}
and we call the soultions of \eqref{eq:4.4} \textbf{K{\"a}hler-Ricci solitons}.
Applying $L_V$ to the both hand sides of  \eqref{eq:4.2}, we have
\begin{equation}
-\Delta_{\partial} \theta_V + \theta_V + V(\kappa) = \text{(const)} \label{eq:4.5} .
\end{equation}
The following function is known as the obstruction to the existence of K{\"a}hler-Ricci solitons:
\begin{theorem}[\cite{TZ02}]
\label{the:4.1}
The function ${\rm {\bf TZ}}_X$ defined over the space of all holomorphic vector fields on $M$ by
\begin{equation}
{\rm {\bf TZ}}_X (V) = \int \theta_V e^{\theta_X} \frac{\omega^m}{m!} = - \int V(\kappa - \theta_X) e^{\theta_X} \frac{\omega^m}{m!} \label{eq:4.6}
\end{equation}
for a holomorphic vector field $X$ is independent of the choice of a K{\"a}hler form $\omega \in 2 \pi c_1 (M)$, here for any $V$, $\theta_V$ is normalized by $-\Delta_{\partial} \theta_V + \theta_V + V(\kappa) = 0$.
\end{theorem}
Now let $M$ be an $m \left( := \sum_{a \in \hat{\mathcal A}} d_a +1 \right)$-dimensional admissible bundle and $\Omega$ an admissible class on $M$. First, we will review the method of constructing GQE metrics over admissible bundles studied by Maschler and T{\o}nnesen-Friedman \cite{MT11}. Let $C^{\infty}([-1,1])$ be the space of smooth functions over the interval $[-1,1]$. According to \cite[Lemma 2.1]{KS86}, there is a one to one correspondence
\[
C^{\infty}([-1,1]) \rightarrow \{ \text{smooth function over $M$ depending only on $z$} \}
\]
given by $S \mapsto S(z):=S \circ z$ for $S \in C^{\infty}([-1,1])$.
\begin{lemma}[\cite{MT11}, Proposition 3.1]
\label{lem:4.2}
For any admissible metric and $S \in C^{\infty}([-1,1])$, we have
\begin{equation}
\Delta_{\partial} S = - \frac{[S'(z) \cdot F(z)]'}{2 p_c(z)} \label{eq:4.7} .
\end{equation}
\end{lemma}
According to \cite[Section 2.2]{ACGT08}, we can calculate the scalar curvature of any admissible metric $g$ as
\begin{equation}
{\rm Scal}_g (\omega) = \frac{1}{2} \left( \sum_{a \in \hat{\mathcal A}} \frac{2d_a s_a x_a}{1+x_a z} -\frac{F''(z)}{p_c(z)} \right) \label{eq:4.8} .
\end{equation}
By \eqref{eq:4.3} and \eqref {eq:4.8}, $\Delta_{\partial} \kappa$ is a function depending only on $z$. Hence \cite[Corollary 3.2]{MT11} implies $\kappa$ depends only on $z$. We can write $\kappa$ as the composition of $z$ and an element of $C^{\infty}([-1,1])$, which we also denote by $\kappa$. On the other hand, \cite[Theorem 4.4]{Kob95} implies that a K{\"a}hler metric $g$ is GQE if and only if its Ricci potential $\kappa$ is a Killing potential. Hence we have
\begin{lemma}
\label{lem:4.3}
An admissible metric $g$ is GQE if and only if there exists $k \in {\mathbb R}$ such that $\kappa = k z$ up to an additive constant.
\end{lemma}
Put
\begin{equation}
P(t) = 2 \int_{-1} ^t \left( \left( \sum_{a \in \hat {\mathcal A}} \frac{d_a s_a x_a}{1+x_a s} \right) \cdot p_c (s) - \frac{\beta_0}{\alpha_0} \cdot p_c (s) \right) ds + 2p_c(-1) \label{eq:4.9} ,
\end{equation}
where $\alpha_0$ and $\beta_0$ are constants defined by
\begin{equation}
\alpha_0 = \int_{-1}^1 p_c(t) dt \;\; \text{and} \;\; \beta_0 = p_c(1) + p_c(-1) + \int_{-1}^1 \left( \sum_{a \in \hat{\mathcal A}} \frac{d_a s_a x_a}{1+x_a t} \right) p_c (t) dt \label{eq:4.10} .
\end{equation}
We often use the following properties for $P(t)$:
\begin{lemma}[\cite{MT11}, Lemma 4.3]
\label{lem:4.4}
For any given admissible data, $P(t)$ satisfies: If $d_0 = 0$, then $P(-1) >0$, otherwise $P(-1) = 0$. If $d_{\infty} = 0$, then $P(1)<0$, otherwise $P(1) = 0$. Furthermore, $P(t)>0$ in some (deleted) right neighborhood of $t= -1$, and $P(t) < 0$ in some (deleted) left neighborhood of $t=1$. Concretely, we see that if $d_0>0$, then $P^{(d_0)} (-1) > 0$ (and the lower order derivatives vanish), while if $d_{\infty} > 0$, then $P^{(d_{\infty})}(1)$ has sign $(-1)^{d_{\infty} + 1}$ (and the lower order derivatives vanish).
\end{lemma}
Combining \eqref{eq:4.3}, \eqref{eq:4.7}, \eqref{eq:4.8}, $\kappa = k z$ and $\overline{\rm Scal } = \beta_0 / \alpha_0$ (cf. \cite[Section 2.2]{ACGT08}), we have
\begin{lemma}
For any admissible GQE metric with the Ricci potential $kz$, the equation
\begin{equation}
F''(z) + k F'(z) = P'(z) \label{eq:4.11}
\end{equation}
holds.
\end{lemma}
We can give the explicit solution for \eqref{eq:4.11} by
\begin{equation}
F(z) = e^{-kz} \int_{-1}^z P(t) e^{kt} dt \label{eq:4.12}
\end{equation}
under the boundary condition $F(-1) =0$ and $F'(\pm 1) = \mp 2 p_c(\pm 1)$. In order to get a GQE metric defined over $M$, $F$ must satisfy $F(1) = 0$, so,
\begin{equation}
{\rm {\bf MT}}(k) := \int_{-1}^1 P(t) e^{kt} dt \label{eq:4.13}
\end{equation}
is an obstruction to the existence of admissible GQE metrics with the Ricci potential $kz$.
\begin{remark}
\label{rem:4.1}
Clearly, $\alpha_0$, $\beta_0$, $P(t)$ and ${\rm {\bf MT}}(k)$ are independent of the choice of admissible metrics $(g, \omega)$. These quantities depend only on $M$ and the admissible class $\Omega$.
\end{remark}
\begin{lemma}
\label{lem4.6}
For any admissible metric, the equation
\begin{equation}
F'(z) + \kappa'(z) \cdot F(z) = P(z) \label{eq:4.14}
\end{equation}
holds.
\end{lemma}
\begin{proof}
By \eqref{eq:4.3}, \eqref{eq:4.7} and \eqref{eq:4.8}, we have
\begin{equation}
\frac{1}{2} \left( \sum_{a \in \hat{\mathcal A}} \frac{2d_a s_a x_a}{1+x_a z} -\frac{F''(z)}{p_c(z)} \right) - \frac{\beta_0}{\alpha_0} = \frac{[ \kappa'(z) \cdot F(z)]'}{2 p_c(z)} \label{eq:4.15} .
\end{equation}
Multiplying  $2p_c(z)$ to the both hand sides of \eqref{eq:4.15} and integrating on $[-1, z]$, we get
\[
F'(z) + \kappa'(z) \cdot F(z) = P(z) + (\text{const}) .
\]
Since $F'(z)$ and $P(z)$ have the same boundary condition, we have $(\text{const}) = 0$. 
\end{proof}
For any $k \in {\mathbb R}$, let $X_{J}^k$ be a holomorphic vector field with the potential function $kz$, i.e., $X_{J}^k$ satisfies $i_{X_{J}^k} \omega = \sqrt{-1} \bar{\partial}_J kz$, where $J$ is the compatible complex structure induced by an admissible metric.
Since $K$ is the infinitesimal generator of the $U(1)$-action on $M$ and the function $z$ is the moment map of this action, we get $i_K \omega = -dz$. Hence $X_{J}^2 =-JK- \sqrt{-1}K$ and $X_{J}^k = \frac{k}{2} \cdot X_{J}^2 = - \frac{k}{2} (JK + \sqrt{-1}K)$.
\begin{theorem}
\label{the:4.2}
Let $M$ be an admissible bundle and $\Omega$ an admissible class with the admissible data $\{x_a\}$. We assume that $\Omega$ coincides with $2 \pi c_1 (M)$ up to a multiple positive constant. (Hence $M$ is Fano and $c_1(M)$ becomes admissible up to scale automatically.)  Then the following statements hold:\\
(1) If we set $ \Omega = 2 \pi \lambda^{-1} c_1 (M)$ for some (positive) constant $\lambda$, then we have $\lambda = \frac{d_0 + d_{\infty} + 2}{2}$.\\
(2) Tian-Zhu's holomorphic invariant \eqref{eq:4.6} and Maschler-T{\o}nnesen's invariant \eqref{eq:4.13} have a relation
\begin{equation}
{\rm {\bf TZ}}_{ \lambda^{-1} X_{J}^k} (X_{J}^2) = -2 \pi \lambda^m \exp{ \left( - \frac{kC}{2 \lambda} \right )} {\rm Vol} \left( S, \prod_{a \in \hat{\mathcal A}} \frac{\omega_a}{x_a}  \right) {\rm {\bf MT}}(k) \label{eq:4.16}
\end{equation}
as a function of $k$, where $C = d_0 - d_{\infty}$.
\end{theorem}
\begin{proof}
In this proof, we consider a fixed admissible metric $g$ whose K{\"a}hler form $\omega$ belongs to $\Omega$.

(1)
Put $g' = \lambda g$ and $\omega' = \lambda \omega$, then $(g', \omega')$ defines a K{\"a}hler structure and $\omega' \in 2 \pi c_1 (M)$. Let $\kappa$ be the Ricci potential of $\omega$. Since the Ricci form is preserved under scaling of $\omega$, $\kappa$ is also the Ricci potential of $\omega'$. In this proof, we promise that $\theta_V$ denotes the potential function of a holomorphic vector field $V$ with respect to $g'$, which is normalized by $-\Delta_{\partial, g'} \theta_V + \theta_V + V(\kappa) = 0$, where $\Delta_{\partial, g'}$ is the  $\partial$-Laplacian with respect to $g'$. We set $\theta_{X_{J}^2} = 2 \lambda z -C$ for some constant $C$, then $C$ is calculated by
\begin{equation}
\begin{split}
C &= -2 \Delta_{\partial, g'} \lambda z + 2 \lambda z + \kappa' (z) \cdot \Theta (z) \\
&= - 2 \Delta_{\partial, g} z + 2 \lambda z + \kappa' (z) \cdot \Theta (z) \\
&= \frac{F'(z)}{p_c (z)} + 2 \lambda z + \kappa' (z) \cdot \Theta (z) \label{eq:4.17} ,
\end{split}
\end{equation}
where we used \eqref{eq:4.7} and $X_{J}^2 (\kappa (z) ) = - JK(\kappa (z) ) = -d ( \kappa (z)) (JK) = \kappa' (z) J dz (K) = \kappa' (z) \cdot \Theta (z) $, and denoted the $\partial$-Laplacian with respect to $g$ by $\Delta_{\partial, g}$. In order to find $C$ as above, we take the limit of $z$ to the boundary.
Since
\begin{equation}
\frac{F'(z)}{p_c (z)} = \Theta' (z) + \Theta (z) \cdot \frac{p_c'(z)}{p_c(z)} =  \Theta' (z) + \Theta (z) \cdot \sum_{a \in \hat{\mathcal A}} \frac{x_a d_a}{1 + x_a z} \label{eq:4.18} ,
\end{equation}
using the boundary condition \eqref{eq:3.4} and de l'H\^opital's rule, we get
\[
\begin{split}
\lim_{z \rightarrow 1} \frac{F'(z)}{p_c (z)} &= -2 + \lim_{z \rightarrow 1} \Theta (z) \cdot \frac{ - d_{\infty} }{1 - z} \\
&= -2 -  \lim_{z \rightarrow 1} \frac{\Theta'(z) d_{\infty}}{-1} = -2-2d_{\infty} .
\end{split}
\]
Similarly,
\[
\lim_{z \rightarrow -1} \frac{F'(z)}{p_c (z)} = 2 + 2 d_0 .
\]
Therefore, combining with \eqref{eq:4.17}, we have
\[
C= -2 -2 d_{\infty} + 2 \lambda = 2+ 2 d_0 -2 \lambda .
\]
Hence we get $C = d_0 - d_{\infty}$ and $ \lambda = \frac{d_0 + d_{\infty} + 2}{2}$.

(2)
From the argument in (1), we have $\theta_{X_{J}^2} = 2 \lambda z -C$ and $\theta_{X_{J}^k} = k \lambda z - \frac{kC}{2}$. Hence by \eqref{eq:4.14} and \eqref{eq:4.17}, we have
\begin{equation}
2\lambda z p_c (z) - C p_c(z) + P(z) = 0 \label{eq:4.19} .
\end{equation}
Hence the direct computation shows that
\[
\begin{split}
{\rm {\bf TZ}}_{\lambda^{-1}X_{J}^k}(X_{J}^2) &= \int ( 2 \lambda z - C ) e^{kz - \frac{kC}{2 \lambda }} \frac{(\lambda \omega)^m}{m!} \\
&= \int ( 2 \lambda z - C) e^{kz - \frac{kC}{2 \lambda }} \lambda^m \cdot p_c(z) \left( \bigwedge_{a \in \hat{\mathcal A}} \frac{(\omega _a / x_a)^{d_a}}{d_a !} \right) dz \wedge \theta \\
&= 2 \pi \lambda^m \exp{ \left( - \frac{kC}{2 \lambda} \right )} {\rm Vol} \left( S, \prod_{a \in \hat{\mathcal A}} \frac{\omega_a}{x_a} \right) \int_{-1}^1 ( 2 \lambda z \cdot p_c(z) - C \cdot p_c (z) ) e^{kz} dz \\
&= -2 \pi \lambda^m \exp{ \left( - \frac{kC}{2 \lambda} \right )} {\rm Vol} \left( S, \prod_{a \in \hat{\mathcal A}} \frac{\omega_a}{x_a} \right) \int_{-1}^1 P(z) e^{kz} dz \\
&= -2 \pi \lambda^m \exp{ \left( - \frac{kC}{2 \lambda} \right )} {\rm Vol} \left( S, \prod_{a \in \hat{\mathcal A}} \frac{\omega_a}{x_a}  \right) {\rm {\bf MT}}(k) ,
\end{split}
\]
where we used the equation $\omega^m /m! = p_c (z) \left( \bigwedge_{a \in \hat{\mathcal A}} \frac{(\omega _a / x_a)^{d_a}}{d_a !} \right) dz \wedge \theta$ (cf. \cite[Section 2.2]{ACGT08}). 
\end{proof}
\begin{corollary}
\label{cor:4.1}
We assume the same as above. Then $\Omega = 2 \pi c_1 (M)$ holds if and only if $d_0 = d_{\infty} =0$, i.e., a blow-down occurs. In this case, we have
\begin{equation}
{\rm {\bf TZ}}_{X_{J}^k} (X_{J}^2) = -2 \pi {\rm Vol} \left( S, \prod_{a \in \hat{\mathcal A}} \frac{\omega_a}{x_a}  \right) {\rm {\bf MT}}(k) \label{eq:4.20}
\end{equation}
for any admissible metrics.
\end{corollary}
\begin{proof}
$\Omega = 2 \pi c_1 (M)$ holds if and only if $\lambda = 1$ if and only if $d_0 = d_{\infty} =0$. 
\end{proof}
%=========Section 5===================================================
\section{Modified K{\"a}hler-Ricci flow on Admissible bundles}
\label{sect:5}
Let $M$ be an $m \left( := \sum_{a \in \hat{\mathcal A}} d_a +1 \right)$-dimensional admissible bundle and $\Omega$ an admissible class. We assume that $P(t)$ has exactly one root in the interval $(-1,1)$. Then we have the following properties:
\begin{lemma}[\cite{MT11}, Lemma 4.4] \label{lem:5.1}
If the function $P(t)$ has exactly one root in the interval $(-1,1)$, then there exists a unique $k_0 \in {\mathbb R}$ such that ${\rm {\bf MT}}(k_0) =0$. Moreover, for this $k_0$, the function $F(z)$ defined by \eqref{eq:4.12} satisfies $F >0$ on $(-1,1)$, and an admissible GQE metric is naturally constructed from $F$.
\end{lemma}
The assumption for $P(t)$ is always satisfied when $|x_a|$ is sufficiently small for all $a \in {\mathcal A}$ (cf. \cite[Section 5]{MT11}) or $ \Omega = 2 \pi \lambda^{-1} c_1 (M)$ for a positive constant $\lambda$ determined by Theorem \ref{the:4.2}. Actually, we have
\begin{lemma}
\label{lem:5.2}
If we assume $ \Omega = 2 \pi \lambda^{-1} c_1 (M)$ for a positive constant $\lambda$, we have $P(t) = (C-2 \lambda t) p_c (t)$, where $\lambda$ and $C$ are constants determined by Theorem \ref{the:4.2}. Hence $P(t)$ has exactly one root $t = \frac{C}{2 \lambda}$ in the interval $(-1,1)$, and there exists an admissible K{\"a}hler-Ricci soliton.
\end{lemma}
\begin{proof}
This follows directly from \eqref{eq:4.19} and Lemma \ref{lem:5.1}. 
\end{proof}
\begin{example}[Koiso's Example (cf. \cite{TZ02}, Example 4.1)]
\label{exa:5.1}
We consider an admissible bundle $M := {\mathbb CP} ^{l+1} \# \overline{ {\mathbb CP}^{l+1}} = {\mathbb P}({\mathcal O} \oplus {\mathcal O}(1)) \rightarrow {\mathbb CP} ^l$ for $l \geq 1$. Since $b_2({\mathbb CP}^l) =1$, every K{\"a}hler class on $M$ is admissible up to scale (cf. \cite[Remark 2]{ACGT08}), so $c_1 (M)$ is admissible up to scale. Hence Corollary~\ref{cor:4.1} implies that there exists an admissible class $\Omega$ with the admissible data $x \in (-1,1)$ $(x \neq 0)$ such that $\Omega = 2 \pi c_1 (M)$. Then we have
\[
{\rm {\bf MT}}(0) = -2 \int_{-1}^1 t (1+xt)^l dt = - 4 \sum_{i = 1}^{[(l+1)/2]} \binom{l}{2 i-1} \frac{x^{2i-1}}{2i+1} \neq 0 .
\]
This shows that there exists an admissible K{\"a}hler-Ricci soliton with respect to a non-trivial holomorphic vector field.
\end{example}
As is seen in Section \ref{sect:3}, for any $\Theta \in {\mathcal K}_\omega ^{\rm adm}$, there exists a unique fiber-preserving $U(1)$-equivariant diffeomorphism $\Psi$ satisfying \eqref{eq:3.7}. Thus we have an inclusion map
\begin{equation}
{\mathcal K}_\omega ^{\rm adm} \hookrightarrow \{ \text{K{\"a}hler form in $(\Omega ,J_c)$}\} \label{eq:5.1}
\end{equation}
defined by $\Theta \mapsto \Psi^* \omega$, where $(\Omega ,J_c)$ denotes the Dolbeault cohomology class with respect to $J_c$.

First, we consider the case of ${\rm {\bf MT}}(0) = 0 $ for simplicity. In this case, there exists an admissible CSC K{\"a}hler metric in $\Omega$. We consider the evolution equation
\begin{equation}
\frac{\partial}{\partial t} \Psi_t ^* \omega = - {\rm Ric}(\Psi_t ^*\omega) + {\mathbb H}{\rm Ric} (\Psi_t ^* \omega) \label{eq:5.2} ,
\end{equation}
where $\Psi_t$ is a $t$-dependent diffeomorphism defined by \eqref{eq:3.7}.
Let $g_t$ be a $t$-dependent admissible metric and $J_t$ be the compatible complex structure corresponding to $\Psi_t$. Taking the trace of the both hand sides of \eqref{eq:5.2}, we have
\begin{equation}
\frac{\partial}{\partial t} {\rm log} \det (\Psi_t ^* g_t) = -{\rm Scal}(\Psi_t ^* g_t) + \overline{\rm Scal} \label{eq:5.3} .
\end{equation}
Since $\Psi_t : (M,J_c,\Psi_t^*g_t,\Psi_t ^* \omega) \rightarrow (M,J_t,g_t,\omega)$ is a biholomorphic isometry, $\Psi_t$ commutes with ${\rm log} \det g_t $ and ${\rm Scal}(g_t)$. Thus we obtain
\begin{equation}
\frac{\partial}{\partial t} \Psi_t ^* {\rm log} \det g_t = \Psi_t ^* ( - {\rm Scal}( g_t ) + \overline{\rm Scal}) \label{eq:5.4} .
\end{equation}

Now we compute a local Ricci potential ${\rm log} \det g_t $ in a local trivialization of $M^0$. This computation is a special case of \cite[(77)]{ACG06} and essentially same as \cite[Lemma 1.2]{KS86}: we take a local trivialization $( \{w^a\}_{a \in \hat{\mathcal A}}, w )$ of the ${\mathbb C}^*$-bundle $M^0 \rightarrow \hat{S}$ such that $w^a = ( w^{a, 1}, \cdots , w^{a, d_a} )$ is a local coordinate system of $S_a$ for each $a \in \hat{\mathcal A}$ and $\frac{\partial}{\partial w} = -J_t K- \sqrt{-1}K$. Then we have
\[
g_t \left( \frac{\partial}{\partial w}, \frac{\partial}{\partial \bar{w}} \right) = 2 \Theta_t (z) , \;
g_t \left( \frac{\partial}{\partial w^{a, i}}, \frac{\partial}{\partial \bar{w}} \right) = 2 \frac{\partial z}{\partial w^{a, i}} ,
\]
\[
g_t \left( \frac{\partial}{\partial w^{a, i}}, \frac{\partial}{\partial w^{b, \bar{j}}} \right) = \frac{2}{\Theta_t (z)} \cdot \frac{\partial z}{\partial w^{a, i}} \cdot \frac{\partial z}{\partial w^{b, \bar{j}}} \; \; ( a \neq b )
\]
and
\[
\begin{split}
g_t \left( \frac{\partial}{\partial w^{a, i}}, \frac{\partial}{\partial w^{a, \bar{j}}} \right) = \frac{1+ x_a z}{x_a} g_a \left( \frac{\partial}{\partial w^{a, i}}, \frac{\partial}{\partial w^{a, \bar{j}}} \right) \\
+ \frac{2}{\Theta_t (z)} \cdot \frac{\partial z}{\partial w^{a, i}} \cdot \frac{\partial z}{\partial w^{a, \bar{j}}} ,
\end{split}
\]
where $i,j = 1, \cdots d_a ; a \in \hat{\mathcal A}$ and $a, b \in \hat{\mathcal A}$. Hence we can compute ${\rm log} \det g_t $ as
\begin{equation}
\begin{split}
{\rm log} \det g_t &= {\rm log} \left( 2 \Theta_t (z) \cdot p_c (z) \cdot \prod_{a \in \hat{\mathcal A}} \det ( g_a / x_a) \right) \\
&= {\rm log} \Theta_t (z) + {\rm log} p_c (z) + \sum_{a \in \hat{\mathcal A}} {\rm log} \det ( g_a / x_a) \label{eq:5.5} .
\end{split}
\end{equation}

Let $V_t$ be the $t$-dependent real vector field corresponding to the $t$-dependent diffeomorphism $\Psi_t$. Then the left hand side of \eqref{eq:5.4} is computed by
\begin{equation}
\begin{split}
\frac{\partial}{\partial t} \Psi_t ^* {\rm log} \det g_t &= \frac{\partial}{\partial t} \Psi_t ^* \left( {\rm log} \Theta_t (z) + {\rm log} p_c (z) + \sum_{a \in \hat{\mathcal A}} {\rm log} \det ( g_a / x_a ) \right) \\
&=  \Psi_t ^* \left( V_t ( {\rm log} \Theta_t (z) ) + \frac{\partial}{\partial t} {\rm log} \Theta_t (z) + V_t( {\rm log} p_c (z) ) \right) \\
&=  \Psi_t ^* \left( \frac{\Theta_t ' (z)}{\Theta_t (z)} \cdot V_t (z) + \frac{1}{\Theta_t (z)} \cdot \frac{d \Theta_t }{dt} (z) + \frac{p_c' (z)}{p_c (z)} \cdot V_t (z)  \right) , \label{eq:5.6}
\end{split}
\end{equation}
where $\frac{d}{d t}$ denotes the partial derivative in $t$ for a function of $z$ and $t$.
\begin{lemma}
\label{lem:5.3}
The equation
\begin{equation}
V_t (z) = - \Theta_t (z) \cdot \frac{d y_t}{dt} (z) \label{eq:5.7}
\end{equation}
holds, where $y_t$ is the function with respect to $\Theta_t$ defined by \eqref{eq:3.7}.
\end{lemma}
\begin{proof}
Differentiating \eqref{eq:3.7} in $t$ implies
\[
V_t (y_t (z) ) + \frac{d y_t}{dt} (z) = 0 .
\]
Since $d_{J_t} ^c (y_t (z)) = y_t ' (z) J_t dz = \theta$, we obtain $dz = - \Theta_t (z) J_t \theta = \Theta_t (z) d (y_t (z) ) $. Hence we have $V_t (z) =dz(V_t) = \Theta_t (z) V_t (y_t (z) ) = \Theta_t (z) \cdot \left( - \frac{d y_t}{d t} (z) \right) = - \Theta_t (z) \cdot \frac{d y_t}{d t} (z)$. 
\end{proof}
Differentiating \eqref{eq:5.7} in $z$, we have
\begin{equation}
\begin{split}
(V_t (z) )' &= - \Theta_t ' (z) \cdot \frac{d y_t}{dt} (z) - \Theta_t (z) \cdot \frac{d}{dt} \left( \frac{1}{\Theta_t} \right) (z) \\
&= - \Theta_t ' (z) \cdot \frac{d y_t}{dt} (z) + \frac{1}{\Theta_t (z)} \cdot \frac{d \Theta_t }{dt} (z) \label{eq:5.8} .
\end{split}
\end{equation}
From \eqref{eq:5.6}, \eqref{eq:5.7} and \eqref{eq:5.8}, we obtain
\begin{equation}
\frac{\partial}{\partial t} \Psi_t ^* {\rm log} \det g_t = \Psi_t ^* \left( (V_t (z))' + \frac{p_c ' (z)}{p_c (z)} \cdot V_t (z) \right) \label{eq:5.9} .
\end{equation}
From \eqref{eq:4.8}, \eqref{eq:5.4}, \eqref{eq:5.9} and  $\overline{\rm Scal} = \beta_0 / \alpha_0$, we get
\[
(V_t (z))' + V_t (z) \cdot \frac{p_c '(z)}{p_c(z)} = - \frac{1}{2} \left( \sum_{a \in \hat{\mathcal A}} \frac{2d_a s_a x_a}{1+x_a z} -\frac{F_t''(z)}{p_c(z)} \right) + \frac{\beta_0}{\alpha_0} ,
\]
where $F_t (z) = \Theta_t (z) \cdot p_c (z)$.
Multipling $2p_c (z)$ and using \eqref{eq:4.9}, we have
\[
2 [V_t (z) p_c(z)]' = - P'(z) + F_t''(z) .
\]
Integrating on the interval $[-1, z]$, this can be written as
\[
2 V_t (z) p_c(z) = - P (z) + F_t'(z) + (\text{const}) .
\]
Since $\Psi_t$ preserves each fiber and fixes any point on the critical set $e_0 \cup e_{\infty}$, we have $V_t (z) \equiv 0$ on $e_0 \cup e_{\infty}$. Moreover, $P$ and $F ' _t$ have the same boundary condition, which yields that
\begin{equation}
2 V_t (z) p_c(z) = - P (z) + F_t'(z) \label{eq:5.10} .
\end{equation}
This is a PDE for a $t$-dependent function $\Theta_t \in {\mathcal K}_\omega ^{\rm adm}$ defined on $[-1,1]$, which is equivalent to \eqref{eq:5.2}.

Now we consider the general case. Let $k_0$ be a real constant such that ${\rm \bf{MT}} (k_0) = 0$. Then there exists an admissible GQE metric $\Theta_{\infty}$ with respect to a holomorphic vector field $X_{J_{\infty}} ^{k_0} = - \frac{k_0}{2} ( J_{\infty} K + \sqrt{-1} K)$, where $J_{\infty}$ denotes the compatible complex structure with $\Theta_{\infty}$. We will also use the notation $\infty$ for the quantities corresponding to $\Theta_{\infty}$ ($\Psi_{\infty}$, $g_{\infty}$, $F_{\infty}$, etc.). Then $\Psi_{\infty}^* g_{\infty}$ is a GQE metric with respect to a holomorphic vector field $X_{J_c} ^{k_0} := \Psi_{\infty} ^{-1} {}_* X_{J_{\infty}} ^{k_0} = - \frac{k_0}{2} ( J_c K + \sqrt{-1} K)$. We consider the evolution equation
\begin{equation}
\frac{\partial}{\partial t} \Psi_t ^* \omega = - {\rm Ric}(\Psi_t ^*\omega) + {\mathbb H}{\rm Ric} (\Psi_t ^* \omega) + L_{X_{J_c} ^{k_0} } \Psi_t ^* \omega \label{eq:5.11} .
\end{equation}
Taking the trace of both hand sides of \eqref{eq:5.11} with respect to $\Psi_t ^* g_t$ yields
\[
\frac{\partial}{\partial t} {\rm log} \det \Psi_t ^* g_t = -{\rm Scal}(\Psi_t ^* g_t ) + \overline{\rm Scal} + \left( L_{X_{J_c} ^{k_0} } \Psi_t ^* \omega , \Psi_t ^* \omega \right)_{\Psi_t ^* g_t} .
\]
Hence we have
\begin{equation}
\frac{\partial}{\partial t} \Psi_t ^* {\rm log} \det g_t = \Psi_t ^* \left( - {\rm Scal}( g_t ) + \overline{\rm Scal} + \left( L_{X_{J_t} ^{k_0} } \omega , \omega \right)_{g_t} \right) \label{eq:5.12} ,
\end{equation}
where we put $X_{J_t} ^{k_0} := \Psi_t {}_* X_{J_c} ^{k_0} =  - \frac{k_0}{2} ( J_t K + \sqrt{-1} K)$. On the other hand,
\[
\begin{split}
L_{X_{J_t} ^{k_0} } \omega &= -\frac{k_0}{2} di_{J_t K} \omega = \frac{k_0}{2} d(\Theta_t (z) \cdot \theta) \\
&= \frac{k_0}{2} ( \Theta_t ' (z) dz \wedge \theta + \Theta_t (z) \cdot d \theta ) .
\end{split}
\]
Hence we can calculate $\left( L_{X_{J_t} ^{k_0} } \omega , \omega \right)_{g_t}$ as
\begin{equation}
\begin{split}
\left( L_{X_{J_t} ^{k_0} } \omega , \omega \right)_{g_t} &= \frac{k_0}{2} ( \Theta_t ' (z) dz \wedge \theta + \Theta_t (z) \cdot d \theta , \omega )_{g_t} \\
&= \frac{k_0}{2} \Theta_t ' (z) ( dz \wedge \theta , \omega )_{g_t} + \frac{k_0}{2} \Theta_t (z) \left( \sum_{a \in \hat{\mathcal A}} \omega_a , \omega \right)_{g_t} \\
&= \frac{k_0}{2} \Theta_t ' (z) + \frac{k_0}{2} \Theta_t (z) \cdot \sum_{a \in \hat{\mathcal A}} \frac{x_a d_a}{1+x_a z} \\
&= \frac{k_0}{2} \Theta_t ' (z) + \frac{k_0}{2} \Theta_t (z) \cdot \frac{p_c ' (z)}{p_c (z)} \label{eq:5.13} .
\end{split}
\end{equation}
Combining \eqref{eq:5.9}, \eqref{eq:5.12} and \eqref{eq:5.13}, we obtain
\[
2 [V_t (z) p_c(z)]' = - P'(z) + F_t''(z) + k_0 F_t ' (z) .
\]
Since $F_t (\pm 1) = 0$, we get the equation
\[
2V_t (z) p_c(z) = - P(z) + F_t'(z) + k_0 F_t (z) .
\]
Summarizing the above, we obtain the following theorem:
\begin{theorem}
\label{the:5.1}
Let $M$ be an $m \left( := \sum_{a \in \hat{\mathcal A}} d_a +1 \right)$-dimensional admissible bundle and $\Omega$ an admissible class on $M$. We assume that $P(t)$ has exactly one root in the interval $(-1,1)$. Then for any fixed symplectic form $\omega$ defined by \eqref{eq:3.3}, the modified K{\"a}hler-Ricci flow \eqref{eq:5.11} can be reduced to
\begin{equation}
2V_t (z) p_c(z) = - P(z) + F_t'(z) + k_0 F_t (z) \label{eq:5.14}
\end{equation}
for $\Theta_t \in {\mathcal K}_\omega ^{\rm adm}$. Here $F_t (z) = \Theta_t (z) \cdot p_c (z)$ and $V_t$ is the $t$-dependent real vector field corresponding to the $t$-dependent fiber-preserving $U(1)$-equivariant diffeomorphism $\Psi_t$ defined by \eqref{eq:3.7}, and $V_t (z)$ is calculated by \eqref{eq:5.7}.
\end{theorem}
We define a $t$-dependent function $\varphi_t$ by $\Theta_t = (1+ \varphi _t ) \Theta_{\infty}$. Combining \eqref{eq:5.14} with $F_{\infty} ' (z) + k_0 F_{\infty} (z) = P(z)$, we get
\begin{equation}
\begin{split}
2 \Theta_{\infty} \frac{d \varphi_t}{dt} &= \Theta_{\infty} \Theta_t \varphi _t'' - (\Theta_{\infty} \varphi_t ' )^2 + \frac{P}{p_c} \cdot \Theta_{\infty} \varphi_t ' \\
&+ \left( \Theta_{\infty} \left( \frac{P}{p_c} \right) '  - \left( \frac{P}{p_c} \right) \Theta_{\infty} ' \right) (1+ \varphi_t) \varphi_t \label{eq:5.15} ,
\end{split}
\end{equation}
where we remark that $\frac{P}{p_c } = - 2 \Delta_{\partial} z + \kappa' \cdot  \Theta $ is smooth on $[-1,1]$.
Guan \cite[Section 11]{Guan07} studied the modified K{\"a}hler-Ricci flow on a certain class of completions of ${\mathbb C}^*$-bundles introduced by Koiso and Sakane \cite{KS86}, and derived the evolution equation of the same type as \eqref{eq:5.15}. He also suggested that under the condition
\begin{equation}
\Theta_{\infty} \left( \frac{P}{p_c} \right) '  - \left( \frac{P}{p_c } \right) \Theta_{\infty} ' < 0 \; \text{on} \; [-1,1], \label{eq:5.16}
\end{equation}
any convergent solution of \eqref{eq:5.15} decays to $0$ in exponential order. Although we can prove the desired result by the maximum principle as in \cite{Koi90}, it is a difficult problem to check whether $\Omega$ satisfies \eqref{eq:5.16} or not in general cases. Hence we consider this problem only in some special situations.
\begin{lemma}
\label{lem:5.4}
We assume that $P(t)$ has exactly one root in the interval $(-1,1)$ and $(\log|P|)'' < 0$ on the complement of the zero-set of $P$ in $(-1,1)$. Then \eqref{eq:5.16} holds.
\end{lemma}
\begin{proof}
Put
\[
\xi (t) = P(t) e^{k_0 t} \;\; \text{and} \;\; \eta (t) = \int_{-1}^t \xi (s) ds .
\]
Then we have $\Theta_{\infty} = \frac{e^{-k_0 z}}{p_c} \eta $ and
\[
\left( \Theta_{\infty} \left( \frac{P}{p_c} \right) '  - \left( \frac{P}{p_c } \right) \Theta_{\infty} ' \right) \cdot e^{k_0 z} p_c = - ( \xi^2 - \eta \cdot \xi' ) \cdot \frac{e^{-k_0 z}}{p_c } .
\] 
By de l'H\^opital's rule, we obtain $\Theta_{\infty} \left( \frac{P}{p_c} \right) '  - \left( \frac{P}{p_c } \right) \Theta_{\infty} ' = -4 (d_{\infty} +1 ) <0$ at $t=1$ and $\Theta_{\infty} \left( \frac{P}{p_c} \right) '  - \left( \frac{P}{p_c } \right) \Theta_{\infty} ' = -4(d_0 + 1) < 0$ at $t=-1$. Hence it suffices to prove that $\xi^2 - \eta \xi' > 0$ on $(-1,1)$. Let $t=t_0$ be the unique root of $P(t)$ in $(-1,1)$, then $\xi(t_0)=0$. Since $P(t)$ has exactly one root in $(-1,1)$, we have $\xi' (t_0) < 0$, $\eta >0$ on $(-1,1)$ and $\eta= 0$ at $t = \pm 1$. Hence we obtain $\xi^2 - \eta \xi' > 0$ at $t= t_0$. Hence we may consider only on the interval $(t_0, 1)$ (a similar proof works on $(-1, t_0)$). One can prove the desired result by the same argument as in \cite[Lemma 3.1]{Koi90}. 
\end{proof}
\begin{remark}
\label{rem:5.1}
If $P(t)$ is a product of polynomials of first order, clearly we have $(\log|P|)'' < 0$ on the complement of the zero-set of $P$ in $(-1,1)$.
\end{remark}
Now we give examples of admissible classes which satisfies \eqref{eq:5.16}.
\begin{example}
\label{exa:5.2}
We assume that $\Omega := 2 \pi \lambda^{-1} c_1 (M)$ is admissible. Then by Lemma~\ref{lem:5.2}, we have $P(t) = (C-2 \lambda t) p_c (t)$ and $(\log|P|)'' < 0$ holds on the complement of the zero-set of $P$ in $(-1,1)$. Hence $\Omega$ satisfies \eqref{eq:5.16} by Lemma~\ref{lem:5.4}.
\end{example}
\begin{example}
\label{exa:5.3}
Let $\Omega$ be an admissible class on $M$ with the admissible data $\{ x_a \}$. Then $\Omega$ satisfies \eqref{eq:5.16} if $|x_a|$ is sufficiently small for all $a \in {\mathcal A}$.
\end{example}
This statement follows from Lemma~\ref{lem:5.4} and the next lemma.
\begin{lemma}
\label{lem:5.5}
Let $\Omega$ be an admissible class on $M$ with the admissible data $\{ x_a \}$. Then $( \log |P|)'' $ is negative on the complement of the zero-set of $P$ in $(-1,1)$ if $|x_a|$ is sufficiently small for all $a \in {\mathcal A}$.
\end{lemma}
\begin{proof}
We denote the limit $x_a \rightarrow 0$ for all $a \in {\mathcal A}$ by $\lim$ for simplicity. We remark that $\lim$ and the derivatives of arbitrary order for $P$ are commutative because the $i$-th derivative $P^{(i)}$ converges uniformly on any closed interval in ${\mathbb R}$ for all $i \geq 1$. From the argument in \cite[Section 5]{MT11}, we can write $\lim P(t)$ as
\[
\lim P(t) = - (2+ d_0 + d_{\infty})(t - t_0) (1+t)^{d_0} (1-t)^{d_{\infty}}
\]
for some $t_0 \in (-1,1)$.
This is a product of polynomials of first order, so we obtain
\[
\frac{(\lim P)'' \cdot \lim P - \{ ( \lim P)' \}^2}{( \lim P)^2} = ( \log| \lim P|)'' < 0
\]
on $(-1, t_0) \cup (t_0, 1)$. Moreover, $ \lim P (t_0) = 0$ and $\lim P' (t_0) < 0$ yield $( \lim P)'' \cdot \lim P - \{ ( \lim P)' \}^2 < 0$ at $t=t_0$. Hence we get $\lim (P'' P -(P')^2) = (\lim P)'' \cdot \lim P - \{ ( \lim P)' \}^2 < 0$ on $(-1,1)$.

We want to show that $P'' P -(P')^2 < 0$ on $(-1,1)$ if $|x_a|$ is sufficiently small for all $a \in {\mathcal A}$. To do this, we observe the behavior of the function $ P'' P -(P')^2$ near the boundary as $x_a \rightarrow 0$ for all $a \in {\mathcal A}$.
\vspace{0.3 cm}

\noindent
Case 1 : $d_0 = 0$

In this case, $\lim P(t)$ has the form
\[
\lim P(t) = - (2+ d_{\infty})(t - t_0) (1-t)^{d_{\infty}} .
\]
From the boundary condition $\lim P(-1) = 2^{d_{\infty} +1}$, $t_0$ is determined by the equation $(2+d_{\infty})(1+t_0) = 2$. Hence the direct computation shows that
\[
\lim (P'' P -(P')^2) = - (1 + d_{\infty}) (4 + d_{\infty}) 2^{2 d_{\infty}} < 0
\]
at $t=-1$. Thus $P'' P -(P')^2$ is negative near $t=-1$ if $|x_a|$ is sufficiently small for all $a \in {\mathcal A}$.
\vspace{0.3 cm}

\noindent
Case 2 : $d_0 = 1$

In this case, we have $\lim P(-1) = 0$ and $\lim P' (-1) > 0 $. Hence $\lim (P'' P -(P')^2)$ is negative at $t=-1$.
This implies that $P'' P -(P')^2$ is negative near $t=-1$ if $|x_a|$ is sufficiently small for all $a \in {\mathcal A}$.
\vspace{0.3 cm}

\noindent
Case 3 : $d_0 \geq 2$

In this case, we have $\lim (P'' P -(P')^2) = 0$ at $t=-1$. However, we can see that $P'' P -(P')^2$ is negative in some (deleted) right neighborhood of $t=-1$ if $|x_a|$ is sufficiently small for all $a \in {\mathcal A}$ because $t=-1$ is a zero point of $P'' P -(P')^2$ fixed as $x_a$ changes.
\vspace{0.3 cm}

A similar observation for $ P'' P -(P')^2 $ near $t=1$ follows in the similar way. As above, we conclude that $ P'' P -(P')^2 $ is negative on $(-1,1)$ if $|x_a|$ is sufficiently small for all $a \in {\mathcal A}$ and this completes the proof of Lemma~\ref{lem:5.5}. 
\end{proof}
From the above, we conclude that
\begin{theorem}
\label{the:5.2}
Let $M$ be an $m \left( := \sum_{a \in \hat{\mathcal A}} d_a +1 \right)$-dimensional admissible bundle and $\Omega$ an admissible class on $M$ with the admissible data $\{x_a\}$. We assume that $P(t)$ has exactly one root in the interval $(-1,1)$. Then for any symplectic form defined by \eqref{eq:3.3}, the modified K{\"a}hler-Ricci flow \eqref{eq:5.11} can be reduced to the evolution equation \eqref{eq:5.15} for $\varphi_t$. Moreover, if $|x_a|$ is sufficiently small for all $a \in {\mathcal A}$, any convergent solution $\varphi_t$ of \eqref{eq:5.15} decays to $0$ in exponential order.
\end{theorem}
By Theorem~\ref{the:5.2} and the definition of $\varphi_t$, we see that $\Theta_t$ converges uniformly to $\Theta_{\infty}$ in exponential order. Here we remark that the convergence of the function $\varphi_t$ dose not directly indicate the convergence of the metric $g_t$ in $C^{\infty}$-topology. In order to get the uniform estimates for the higher order derivatives of $g_t$, we need additional argument which substitutes for Cao's estimate for complex Monge-Amp{\`e}re equation (cf. \cite{Cao85}).
%=========References===================================================

\end{document}